\documentclass{llncs}
\usepackage{epsfig}

\newcommand{\R}{{I\!\!R}}

\newcommand{\be}{\begin{equation}}
\newcommand{\ee}{\end{equation}}
\newcommand{\ba}{\begin{eqnarray}}
\newcommand{\ea}{\end{eqnarray}}

\newcommand{\la}{\label}

\def\R{{\rm I}\! {\rm R}}

\def\ss{ss}

\def\X{{\bf X}}
\def\Y{{\bf Y}}

\def\D{{\cal D}}

\usepackage{amsmath,amssymb}
\usepackage{amsfonts}
\usepackage{makeidx}

\begin{document}

\pagestyle{headings}

\title{Iterative operator-splitting methods for unbounded operators: Error analysis and examples}
\author{J\"urgen Geiser}
\institute{
\email{geiser@mathematik.hu-berlin.de}}
\maketitle

\begin{abstract}
In this paper we describe an iterative
operator-splitting method for unbounded operators.
We derive error bounds for iterative splitting methods
in the presence of unbounded operators and semigroup operators.
Here mixed applications of hyperbolic and parabolic type are
allowed and discussed in the applications.
Mixed experiments are applied to  ordinary differential equations 
and evolutionary Schr\"odinger equations.

\end{abstract}



\section{Introduction}

In this paper we concentrate on approximation
to the solution of the linear evolution equation
\begin{eqnarray}
\label{equ1}
&&  \partial_t \; u = Lu = (A + B) u, \; u(0) = u_0 ,
\end{eqnarray}
where $L, A$ and $B$ are unbounded operators.

As numerical method we will apply a two-stage iterative splitting scheme:
\begin{eqnarray}
\label{equ1}
&&   u_i(t) = \exp(A t) u_0 + \int_0^t \exp(As) B u_{i-1} \; ds , \\
&&   u_{i+1}(t) = \exp(B t) u_0 + \int_0^t \exp(Bs) A u_{i} \; ds ,
\end{eqnarray}
where $i = 1, 3, 5, \ldots$ and $u_{0}(t) = 0$.

The outline of the paper is as follows.
The operator-splitting methods are introduced in Section \ref{oper}
and the error analysis of the operator-splitting methods are
presented.
In Section \ref{semi}, we discuss the semigroup theory, underlying 
the theoretical method.
 In Section \ref{error}, we discuss the error
analysis of the iterative methods.
In Section \ref{comp}, we discuss an efficient 
computation of the iterative splitting method with $\phi$-functions.
In Section \ref{appl} we introduce the application of our
methods to existing software tools.
Finally we discuss future works in the area of iterative
splitting methods.

\section{Iterative splitting method}
\label{oper}

The following algorithm is based on the iteration with 
fixed-splitting discretization step-size $\tau$, namely, on the 
time-interval $[t^n,t^{n+1}]$ we solve the following sub-problems
consecutively  for $i=0,2, \dots 2m$. (cf. \cite{glow03,kan03}.):

\begin{eqnarray}
 && \frac{\partial c_i(t)}{\partial t} = A
c_i(t) \; + \; B c_{i-1}(t), \;
\mbox{with} \; \; c_i(t^n) = c^{n} \label{kap3_iter_1} \\
&& \mbox{and} \; c_{0}(t^n) = c^n \; , \; c_{-1} = 0.0 , \nonumber
\\\label{kap3_iter_2}
&& \frac{\partial c_{i+1}(t)}{\partial t} = A c_i(t) \; + \; B c_{i+1}(t), \; \\
&& \mbox{with} \; \; c_{i+1}(t^n) = c^{n}\; , \nonumber
\end{eqnarray}
 where $c^n$ is the known split
approximation at the  time-level  $t=t^{n}$. The split
approximation at the time-level $t=t^{n+1}$ is defined as
$c^{n+1}=c_{2m+1}(t^{n+1})$. (Clearly, the function $c_{i+1}(t)$
depends on the interval $[t^n,t^{n+1}]$, too, but, for the sake of
simplicity, in our notation we omit the dependence on $n$.)

\smallskip
In the following we will analyze the convergence and the rate of
convergence  of the method
(\ref{kap3_iter_1})--(\ref{kap3_iter_2}) for $m$ tends to
infinity for the linear operators 
$A,B: \!  {\X} \rightarrow {\X}$, 
where we assume that these operators and their sum  are
generators of the $C_0$ semigroups. We emphasize that these
operators are not necessarily bounded, so the convergence is
examined in a general Banach space setting.

\section{Semi group theory}
\label{semi}

In the theoretical part, we deal with systems of operators.
Therefore here we have to prove that 
their operators are generators of  $C_0$-semigroup.

This is not trivial and in the following we give the
detail to verify that the generators based on a  graph norm are 
bounded operators (see ideas in \cite{hieb92}).

\subsection{$2 \times 2$ Systems}

We deal with two iterative steps and obtain
 operators of $2 \times 2$ matrices.

Let us assume $\X$ to be a Banach space and let $A$, $B$ be 
generators of a $C_0$-semigroups in $\X$, so we define the system:
\begin{eqnarray}
{\Y_A} = \left( \begin{array}{c}
{\X_A} \\
{\X}
\end{array}
\right) , 
\end{eqnarray}
\begin{eqnarray}
C :=
 \left( \begin{array}{c c}
A & 0  \\
A & B 
\end{array}
\right): {\Y_{A}}  \rightarrow {\Y_{A}} , 
\end{eqnarray}
where $C$ is a $2 \times 2$ matrix operator.

We assume that ${\X}_A$ is the domain of $A$ with the
graph norm $||\cdot||_A$:
\begin{eqnarray}
||g||_A := ||A g|| + ||g||,
\end{eqnarray}
where $g \in dom(A)$.

\begin{theorem}
Assuming $A$ and $B$ are generators of $C_0$ semigroup in
${\X_A}$ and ${\X}$,
we have a closed operator $C$ as a generator of a $C_0$-semigroup
in ${\Y_A}$
\end{theorem}

\begin{proof}

We solve the Cauchy problem:
\begin{eqnarray}
\frac{d}{dt} T(t) f = T(t) C f , \; T(0) f = f, \; f \in dom(C) .
\end{eqnarray}
and we find
\begin{eqnarray}
T(t) f = \left( \begin{array}{c c}
\exp(A t) & 0 \\
\int_0^t \exp(B r) A \exp(A (t-r)) dr & \exp(B t)
\end{array}
\right)
 \left( \begin{array}{c}
 g \\
 h
\end{array}
\right) , 
f :=  \left( \begin{array}{c}
 g \\
 h
\end{array}
\right) , 
\end{eqnarray}
for 
\begin{eqnarray}
f =  \left( \begin{array}{c}
 g \\
 h
\end{array}
\right) \in {\D} :=  \left( \begin{array}{c}
 dom(A) \\
 {\X}
\end{array}
\right) , 
\end{eqnarray}

And we have:
\begin{eqnarray}
\frac{d}{dt} T(t) f & = & \frac{d}{dt} \left( \begin{array}{c}
\exp(A t) \; g  \\
\int_0^t \exp(B r) A \exp(A (t-r)) dr \; g  + \exp(B t) \; h
\end{array}
\right) \\
& = &
 \left( \begin{array}{c}
\exp(A t) A  \; g  \\
\exp(B t) A  \; g + \int_0^t \exp(B r) A^2 \exp(A (t-r)) dr  \; g \\
 + \exp(B t) B \; h
\end{array}
\right)\\
&= &
 \left( \begin{array}{c c}
\exp(A t)  \\
\int_0^t \exp(B r) A \exp(A (t-r)) dr & \exp(B t) 
\end{array}
\right)
 \left( \begin{array}{c c}
A  & 0   \\
A  & B
\end{array}
\right)
\left( \begin{array}{c}
 g \\
 h
\end{array}
\right) , 
\end{eqnarray}
for $g \in dom(A^2)$ and $h \in dom(B)$.

The same can be shown for $T(t) T(s) f = T(t + s) f, f \in {\D}$.

Therefore the family $\{T(t)\}_{t \le 0}$ is a $C_0$ semigroup in ${\Y}_A$,
while $R(t) = \int_0^t \exp(B r) A \exp(A (t-r)) dr$ is defined on $dom(A)$
and can be bounded in ${\X}_A$ for each $t > 0$.

\end{proof}

\begin{remark}
We cannot weaken the assumptions to a the closed operator $C$ as generator of 
a semigroup $\Y =  \left( \begin{array}{c}
{\X} \\
{\X}
\end{array}
\right)$.
This is obvious, while if we set $A = B$ we find
\begin{eqnarray}
T(t) = \left( \begin{array}{c c}
\exp(A t) & 0 \\
t A \exp(At) & \exp(A t)
\end{array}
\right)
\end{eqnarray}
and if $\X$ is a Hilbert space and $i A$ is selfadjoint,
then $R(t) = t A \exp(At)$ cannot be extended to a bounded operator
in $\X$ for each $t > 0$, unless we restrict it to ${\X}_A$ with the
graph norm $|| \cdot ||_A = ||A \cdot || + || \cdot||$.

\end{remark}

\subsection{$N \times N$ Systems (Generalization)}

Here, we deal with $n$ iterative steps and obtain operators 
of $n \times n$ matrices.

Let $\X$ be a Banach space and let $A$, $B$ be generators
of a $C_0$-semigroups in $\X$, so we define the system:
\begin{eqnarray}
{\Y_{ABA\ldots A}} = \left( \begin{array}{c}
{\X_{ABA \ldots A}} \\
{\X_{ABA \ldots B}} \\
\vdots \\
{\X_{A}} \\
{\X}
\end{array}
\right)
\end{eqnarray}
and
\begin{eqnarray}
C :=
 \left( \begin{array}{c c c c}
A & 0 & \ldots & \ldots   \\
A & B & 0 & \ldots  \\
0 & B & A & \ldots \\
\vdots & \ddots & \ddots &\ddots \\
0 & \ldots  & A & B
\end{array}
\right): {\Y_{ABA\ldots A}}  \rightarrow {\Y_{ABA\ldots A}}  
\end{eqnarray}

We assume that ${\X}_{ABA \cdots A}$ is the domain of $A$ with the
graph norm $||\cdot||_{ABA \cdots A}$:
\begin{eqnarray}
||g||_{ABA \cdots A} := ||ABA \cdots A g|| + \ldots + ||A g|| + || g ||,
\end{eqnarray}
where $g \in dom(A)$ and we assume .

\begin{remark}
The proof can be done as for a $2 \times 2$ operators, and we use recursive
results.
\end{remark}

\section{Error analysis}
\label{error}

We present the results of the consistency of our 
iterative method. We assume for the system of operator the 
generator of a $C_0$ semigroup based on their underlying
graph norms (see the previous Section \ref{semi}).

\smallskip
\begin{theorem}
Let us consider the abstract Cauchy problem in a Hilbert space \X
\begin{equation}
\begin{array}{c}
{\displaystyle \partial_t c(x, t) = A c(x, t) + B c(x, t), \quad 0 < t
\leq T } \mbox{and} \; x \in \Omega\\
\noalign{\vskip 1ex} {\displaystyle c(x,0)=c_0(x) }  \; x \in \Omega\\
\noalign{\vskip 1ex} {\displaystyle c(x,t)=c_1(x, t) }  \; x \in \partial \Omega \times [0, T], \\

\end{array} \label{eq:ACP}
\end{equation}

\noindent where  $A,B: \!  D({\X}) \rightarrow {\X} $ are given
linear  operators which are generators of the $C_0$-semigroup and $c_0
\in \X$ is a given element. 
We assume $A$, $B$ are unbounded.
Further, we assume the estimations of the unbounded operator $B$ with 
sufficient smooth initial conditions (see \cite{han08}):
\begin{eqnarray}
&& || B  \exp((A+B) \tau) u_0 || \le \kappa ,
\end{eqnarray}

Further we assume the estimation of the partial 
integration of the unbounded operator $B$ (see \cite{han08}):
\begin{eqnarray}
&& || B \int_0^{\tau} \exp(B s) s ds || \le \tau C , 
\end{eqnarray}

Then, we can bound our iterative operator splitting method as :
\begin{eqnarray}
&& || (S_i -  \exp((A+B) \tau)|| \le  C \tau^i, 
\end{eqnarray}
where $S_i$ is the approximated solution for the i-th iterative step
and $C$ is a constant that can be chosen uniformly on bounded time
intervals.

\end{theorem}

\begin{proof}
Let us consider the  iteration
(\ref{kap3_iter_1})--(\ref{kap3_iter_2}) on the  sub-interval
$[t^n,t^{n+1}]$.

For the first iterations we have:
\begin{equation}
\begin{array}{c}
 \partial_t c_1(t) = A c_1(t) , \quad t \in (t^n,t^{n+1}], 
\end{array} \label{eq:err1}
\end{equation}
and for the second iteration we have:
\begin{equation}
\begin{array}{c}
 \partial_t c_2(t) = A c_1(t)  + B c_2(t) , \quad t \in (t^n,t^{n+1}], \\
\end{array} \label{eq:err1}
\end{equation}

In general, we have:

for the odd iterations: $i = 2m+1$
for $m= 0,1, 2, \ldots$
\begin{equation}
\begin{array}{c}
 \partial_t c_i(t) = A c_i(t) + B c_{i-1}(t) , \quad t \in (t^n,t^{n+1}], 
\end{array} \label{eq:err1}
\end{equation}
where for $c_0(t) \equiv 0$.

for the even iterations: $i = 2m$
for $m= 1, 2, \ldots$
\begin{equation}
\begin{array}{c}
 \partial_t c_{i}(t) = A c_{i-1}(t)  + B c_{i}(t) , \quad t \in (t^n,t^{n+1}] . \\
\end{array} \label{eq:err1}
\end{equation}

We have the following solutions for the iterative scheme:

the solutions for the first two equations are given by the
variation of constants:
\begin{equation}
\begin{array}{c}
 c_1(t) = \exp(A (t^{n+1} - t )) c(t^n) , \quad t \in (t^n,t^{n+1}], \\
\end{array} \label{eq:err1}
\end{equation}
\begin{equation}
\begin{array}{c}
 c_2(t) = \exp(B t) c(t^n)  + \int_{t^n}^{t^{n+1}} \exp(B (t^{n+1} - s)) A c_1(s) ds , \quad t \in (t^n,t^{n+1}] .
\end{array} \label{eq:err1}
\end{equation}

For the recurrence relations with even and odd iterations,
we have the solutions:

for the odd iterations: $i = 2m+1$,

for $m= 0, 1, 2, \ldots$
\begin{equation}
\label{odd_1}
\begin{array}{c}
 c_{i}(t) = \exp(A (t - t^n )) c(t^n) + \int_{t^n}^{t} \exp(s A) B c_{i-1}(t^{n+1} - s) \; ds , \quad t \in (t^n,t^{n+1}] .
\end{array}
\end{equation}

For the even iterations: $i = 2m$,

for $m= 1, 2, \ldots$
\begin{equation}
\label{even_1}
\begin{array}{c}
 c_{i}(t) = \exp(B (t - t^n )) c(t^n) + \int_{t^n}^{t} \exp(s B) A c_{i-1}(t^{n+1} - s) \; ds , \quad t \in (t^n,t^{n+1}] .
\end{array}
\end{equation}

{\bf The consistency is given as:}

For $e_1$ we have:
\begin{eqnarray}
&& c_1(\tau) = \exp(A) \tau) c(t^n) ,
\end{eqnarray}
\begin{eqnarray}
&& c(\tau) = \exp((A+B) \tau) c(t^n) = \exp(A \tau) c(t^n) \\
&& + \int_{t^n}^{t^{n+1}}  \exp(A s) B \exp((t^{n+1} - s) (A + B)) c(t^n) \; ds . \nonumber
\end{eqnarray}

We obtain:
\begin{eqnarray}
&& || e_1 || = || c - c_1 || \le || \exp((A+B) \tau) c(t^n) - \exp(A \tau) c(t^n)|| \\
&& \le C_1 \tau c(t^n) . \nonumber
\end{eqnarray}

For $e_2$ we have:
\begin{eqnarray}
&& c_2(\tau) = \exp(B) \tau) c(t^n) \nonumber \\
&& + \int_{t^n}^{t^{n+1}}  \exp(B s) A \exp((t^{n+1} - s) A) c(t^n) \; ds ,
\end{eqnarray}
\begin{eqnarray}
&& c(\tau) = \exp(B \tau) c(t^n) + \int_{t^n}^{t^{n+1}}  \exp(B s) A \exp((t^{n+1} - s) A) c(t^n) \; ds  \nonumber \\
&& + \int_{t^n}^{t^{n+1}}  \exp(B s) A \\
&& \int_{t^n}^{t^{n+1} - s} \exp(A \rho) B \exp((t^{n+1} - s - \rho) (A + B)) c(t^n) \; d \rho \; ds . \nonumber 
\end{eqnarray}

We obtain:
\begin{eqnarray}
&& || e_2 || \le || \exp((A+B) \tau) c(t^n) - c_2 || \\
&& \le C_2 \tau^{2} c(t^n) . \nonumber 
\end{eqnarray}

For odd and even iterations, the recursive proof is given in the following:

for the odd iterations: $i = 2m+1$

for $m= 0, 1, 2, \ldots$,

for $e_i$ we have :
\begin{eqnarray}
&& c_i(\tau) = \exp(A) \tau) c(t^n)  \\
&& + \int_{t^n}^{t^{n+1}}  \exp(A s) B \exp((t^{n+1} - s) B) c(t^n) \; ds \nonumber\\
&& + \int_{t^n}^{t^{n+1}}  \exp(A s_1) B \int_{t^n}^{t^{n+1} - s_1} \exp(s_2 B) A \exp((\tau - s_1 - s_2) A)  c(t^n) \; ds_2 \; ds_1 \nonumber \\
&& + \ldots + \nonumber \\
&& + \int_{t^n}^{t^{n+1}}  \exp(A s_1) B \int_{t^n}^{t^{n+1} - s_1} \exp(s_2 B) A \exp((\tau - s_1 - s_2) A)  uc(t^n) \; ds_2 \; ds_1 + \ldots + \nonumber \\
&&  + \int_{t^n}^{t^{n+1}}  \exp(A s_1) B \int_{t^n}^{t^{n+1} - \sum_{j=1}^{i-1} s_1} \exp(s_2 B) A \exp((\tau - s_1 - s_2) A)  c(t^n) \; ds_2 \; ds_1 \ldots ds_{i} \nonumber ,
\end{eqnarray}
\begin{eqnarray}
&& c(\tau) = \exp(B \tau) + \int_{t^n}^{t^{n+1}}  \exp(B s) A \exp((t^{n+1} - s) A) c(t^n) \; ds  \\
&& + \ldots + \nonumber \\
&& + \int_{t^n}^{t^{n+1}}  \exp(A s_1) B \int_{t^n}^{t^{n+1} - s_1} \exp(s_2 B) A \exp((\tau - s_1 - s_2) A)  c(t^n) \; ds_2 \; ds_1 + \ldots + \nonumber \\
&&  + \int_{t^n}^{t^{n+1}}  \exp(A s_1) B \int_{t^n}^{t^{n+1} - \sum_{j=1}^{i-1} s_1} \exp(s_2 B) A \exp((\tau - s_1 - s_2) A)  c(t^n) \; ds_2 \; ds_1 \ldots \nonumber \\
&& \int_{t^n}^{t^{n+1} - \sum_{j=1}^{i} s_2} \exp(s_2 B) A \exp((\tau - s_1 - s_2) (A+B))  c(t^n) ds_{i} \nonumber .
\end{eqnarray}

We obtain:
\begin{eqnarray}
&& || e_i || \le || \exp((A+B) \tau) c(t^n) - c_i || \\
&& \le C \tau^{i} c(t^n) \nonumber ,
\end{eqnarray}
where $\alpha = \min_{j=1}^{i}\{ \alpha_1\}$
and $0\le \alpha_i < 1$. \\

The same idea can be applied to the 
even iterative scheme.

\end{proof}

\begin{remark}
The same idea can be applied to 
 $A = \nabla D \nabla$  $B = -{\bf v} \cdot \nabla$,
so that one operator is less unbounded but we reduce the convergence order:
\begin{eqnarray}
 && ||e_1|| = K ||B|| \tau^{\alpha_1} ||e_{0}|| + \mathcal{O}(\tau^{1 + \alpha_1}) \\
&& \mbox{and hence} \nonumber \\
 && ||e_{2}|| = K ||B|| ||e_0|| \tau^{1+\alpha_1+\alpha_2} + \mathcal{O}(\tau^{1+\alpha_1+\alpha}) \label{gleich_kap33b} ,
\end{eqnarray}
where $0 \le \alpha_1, \alpha_2 < 1$.
\end{remark}

\begin{remark}
If we assume the consistency of $\mathcal{O}(\tau^m)$ for the 
initial value $e_1(t^n)$ and $e_2(t^n)$, we can redo the proof and obtain
at least a global error of the splitting methods of  $\mathcal{O}(\tau^{m-1})$.

\end{remark}

In the next section we describe the computation of the integral formulation 
with $\exp$-functions.

\section{Computation of the iterative splitting method}
\label{comp}

In the last few years, the computational effort to compute integral with
$\exp$-function has increased because of the $\phi$-function,
which reduces the integration to a product of $\exp$-functions,
see \cite{han08}. The ideas are also used for 
exponential Runge-Kutta methods, see \cite{hoch05}.

As regards computations of the matrix exponential an
overview is given in \cite{najfeld95}.

For linear operators $A, B: {\D}(F) \subset X \rightarrow X$ generating
a $C_0$ semigroup and a scalar $t \in \R$, we define 
the operator $a = t A$ and $b = t B$, and the bounded 
operators $\phi_{0,A}= \exp(a)$, $\phi_{0,B}= \exp(b)$
and:
\begin{eqnarray}
\phi_{k,A} & = & \int_0^1 \exp((1 - s) \tau A) \frac{s^{k-1}}{(k-1)!} d s) , \\
\phi_{k,B} & = & \int_0^1 \exp((1 - s) \tau B) \frac{s^{k-1}}{(k-1)!} d s) ,
\end{eqnarray}
for $k \ge 1$.

From this definition it is a straightforward matter to prove the recurrence
relation:
\begin{eqnarray}
\label{rec_1}
\phi_{k,A} = \frac{1}{k!} I + \tau A \phi_{k+1}, \\
\label{rec_2}
\phi_{k,B} = \frac{1}{k!} I + \tau B \phi_{k+1} .
\end{eqnarray}

We apply equations (\ref{rec_1}) and (\ref{rec_2}) to
our iterative schemes (\ref{odd_1}) and (\ref{even_1})
and obtain:
\begin{eqnarray}
&& c_1(\tau) = \exp(A \tau) c(t^n)  = \phi_{0,A} c(t^n) , \\
&& c_2(\tau) =  \phi_{0,A} c(t^n) + \sum_{k=1}^{\infty}  B^k A \phi_{k,A}  ,
\end{eqnarray}
where we assume that $B$ is bounded and $\exp{B} = \sum_{k=0}^\infty \frac{1}{k} B^k$.

For an unbounded operator $B$ we can apply the convolution of integrals,
exactly with the Laplacian transformation or numerically with 
integration rules.

\subsection{Exact Computation of the Integrals}

To obtain analytical solutions of the differential equations:
\begin{eqnarray}
\label{equ1_app}
&&  \partial_t c_1 = A c_1 \\
&& \partial_t c_2 = A c_1 + B c_2 \\
&& \vdots \nonumber \\
&& \partial_t c_{i+1} = A c_{i+1} + B c_{i+1} 
\end{eqnarray}
where $c(t^n)$ is the initial condition and $A, B$ are unbounded 
operators.

We apply Laplacian transformation of the 
differential equations respecting the unbounded operators,
see \cite{dav78}. \\

We use the Laplace transformation for the translation
of the ordinary differential equation.
The transformations for this cases are given in \cite{dav78}. For that we need 
to define the transformed function $\hat u = \hat u(s,t)$:
\begin{equation}
\label{eq5}
  \hat u_i(s,t) := \int \limits_0^{\infty} u_i(x,t) \, e^{-s x} \,
  dx \,.
\end{equation}

We obtain the following analytical solution of the first iterative steps
with the re-transformation:
\begin{eqnarray}
\label{equ2_app}
&&  c_1 = \exp(A t) c(t^n) , \\
&&  c_2 = A (B - A)^{-1} \exp(A t) c(t^n)+ A ( A - B)^{-1} \exp(B t) c(t^n) .
\end{eqnarray}

The solutions of the next steps can be done recursively.

The Laplacian transformation is given as :
\begin{eqnarray}
\label{equ1_app}
&&  \tilde{c}_1 = (I s + A)^{-1} c_{01} \\
&& \tilde{c}_2 = (I s + B)^{-1} c_{02} + (I s + B)^{-1} A \tilde{c}_1  \\
&& \tilde{c}_3 = (I s + A)^{-1} c_{03} + (I s + A)^{-1} A \tilde{c}_2  \\
&& \ldots  \nonumber  .
\end{eqnarray}

Here we assume the commutation of

$(A-B)^{-1} A = A (A -B)^{-1}$,

and we can apply the decomposition of partial fraction:
\begin{eqnarray}
\label{equ1_app}
&&  (I s + A)^{-1} A (I s + B)^{-1} =  (I s + A)^{-1} (B-A)^{-1} A  \nonumber \\
&& + A (B - A)^{-1} (I s + B)^{-1} .
\end{eqnarray}

Here we can derive our solutions:
\begin{eqnarray}
\label{equ2_app}
&&  c_2 = \exp(B t) c(t^n) \\
&& + A (B - A)^{-1} \exp(A t) c(t^n) + A ( A - B)^{-1} \exp(B t) c(t^n) \nonumber .
\end{eqnarray}

We have the following recurrent argument for the Laplace-Transformation:

for the odd iterations: $i = 2m+1$

for $m= 0,1, 2, \ldots$
\begin{equation}
\begin{array}{c}
\tilde{c}_i = (I s - A)^{-1} \;  c_n +  (I s - A)^{-1} B \; \tilde{c}_{i-1} , 
\end{array} \label{laplace_1} 
\end{equation}

for the even iterations: $i = 2m$
for $m= 1, 2, \ldots$
\begin{equation}
\begin{array}{c}
 \tilde{c}_{i}(t) = (I s - B)^{-1} A \; \tilde{c}_{i-1}  + (I s - B)^{-1} \;  c_n ,
\end{array} \label{laplace_2}
\end{equation}

We develop the next iterative solution $c_3$ as follows:
\begin{eqnarray}
\label{equ2_app}
&&  c_3 = \exp(A t) c(t^n) \\
&& + B A (B - A)^{-1} t \exp(A t) c(t^n) \nonumber \\
&& + B A ( A - B)^{-1} (B - A)^{-1} \exp(A t) c(t^n)  \nonumber \\
&& + B A ( A - B)^{-1} (A - B)^{-1} \exp(B t) c(t^n) . \nonumber
\end{eqnarray}

We apply the iterative steps recursively and
 obtain for the odd iterative scheme the following
recurrent argument:
\begin{eqnarray}
\label{equ2_app}
&&  c_i = \exp(A t) c(t^n) \\
&& + B A (B - A)^{-1} t \exp(A t) c(t^n)  \nonumber \\
&& + \ldots +  \nonumber \\
&& +  B A \ldots B A (B - A)^{-1} \ldots (B - A)^{-1} t^{i-2} \exp(A t) c(t^n)  \nonumber \\
&& +  B A \ldots B A (B - A)^{-1} \ldots (B - A)^{-1} ( A - B)^{-1} (B - A)^{-1} \exp(A t) c(t^n)  \nonumber \\
&& +  \ldots +  B A \ldots B A (B - A)^{-1} \ldots (B - A)^{-1} (A - B)^{-1} \exp(B t) c(t^n) \nonumber .
\end{eqnarray}

\begin{remark}
The same recurrent argument can be applied to the even iterative scheme.
Here we have only to apply matrix multiplications and can skip the
time-consuming integral computations.
Only two evaluations for the exponential function for $A$ and $B$
are necessary. The main disadvantage of computing the 
iterative scheme exactly is the time-consuming inverse matrices. 
These can be skipped with numerical methods. 
\end{remark}

\subsection{Numerical Computation of the Integrals}

Here our main contributions 
are to skip the integral formulation of the exponential 
functions and to apply only matrix multiplication of given 
exponential functions. Such operators can be computed at the beginning
of the evaluation.

{\bf Evaluation with Trapezoidal rule (two iterative steps).}

We have to evaluate:
\begin{equation}
\begin{array}{c}
 c_2(t) = \exp(B t) c(t^n)  + \int_{t^n}^{t^{n+1}} \exp(B (t^{n+1} - s)) A c_1(s) ds , \quad t \in (t^n,t^{n+1}] , 
\end{array} \label{eq:err1}
\end{equation}
where $c_1(t) = \exp(A t) \exp(Bt) c(t^n)$.

We apply the Trapezoidal rule and obtain:
\begin{equation}
\begin{array}{c}
 c_2(t) = \exp(B t) c(t^n)  + \frac{1}{2} \Delta t \left( B \exp(A t) \exp(B t) + exp(A t) B  \right), 
\end{array} \label{rule1}
\end{equation}
where $c_1(t) = \exp(A t) \exp(Bt) c(t^n)$ and $\Delta t = t - t^n$.

{\bf Evaluation with Simpson rule (three iterative steps).}

We have to evaluate:
\begin{equation}
\begin{array}{c}
 c_3(t) = \exp(A t) c(t^n)  + \int_{t^n}^{t^{n+1}} \exp(A (t^{n+1} - s)) B c_2(s) ds , \quad t \in (t^n,t^{n+1}] , 
\end{array} \label{eq:err1}
\end{equation}
where $c_1(t) = \exp(A \frac{t}{2}) \exp(Bt) \exp(A \frac{t}{2}) c(t^n)$.

We apply the Simpson rule and obtain:
\begin{eqnarray}
\label{rule2}
 c_3(t) & = & \exp(A t) c(t^n)  + \frac{1}{6} \Delta t \left( B  \exp(A \frac{t}{2}) \exp(Bt) \exp(A \frac{t}{2}) \right. \\
& + & \left.  4  \exp(A \frac{t}{2})  B  exp(A \frac{t}{4})  \exp(B \frac{t}{2}) \exp(A \frac{t}{4}) +  \exp(A t) B \right), \nonumber
\end{eqnarray}
where $c_1(t) = \exp(A t) \exp(Bt) c(t^n)$ and $\Delta t = t - t^n$.

\begin{remark}
The same result can also be derived by applying BDF3 (Backward Differential Formula of Third Order).
\end{remark}

{\bf Evaluation with Bode rule (four iterative steps).}

We have to evaluate:
\begin{equation}
\begin{array}{c}
 c_4(t) = \exp(B t) c(t^n)  + \int_{t^n}^{t^{n+1}} \exp(B (t^{n+1} - s)) A c_3(s) ds , \quad t \in (t^n,t^{n+1}] , 
\end{array} \label{eq:err1}
\end{equation}
where $c_3(t)$ has to be evaluated with a third order method.

We apply the Bode rule and obtain:
\begin{eqnarray}
\label{rule3}
 c_4(t) & = & \exp(A t) c(t^n)  + \frac{1}{90} \Delta t \left( 7 A c_3(0) + 32 \exp(B \frac{t}{4}) A c_3( \frac{t}{4}) \right. \\
& + & \left.  12 \exp(B \frac{t}{2}) A c_3( \frac{t}{2}) + 32 \exp(B \frac{3 t}{4}) A c_3( \frac{3 t}{4})+ 7 \exp(B t) A c_3(t) \right), \nonumber
\end{eqnarray}
where $c_3(t)$ is evaluated with the Simpson rule or a further
third order method. We have $\Delta t = t - t^n$.

\begin{remark}
The same result can also be derived by applying the fourth order 
Gauss Runge Kutta method.
\end{remark}

In the next section we describe the numerical results of our methods.

\section{Numerical Examples}
\label{appl}

In the next example, we applied our iterative scheme with
their underlying numerical approximations to differential equations.

\subsection{Linear ordinary differential equation}

We deal with the linear ordinary differential equation:
\begin{eqnarray} 
\frac{\partial u(t)}{\partial t} =  \left(\begin{array}{cc}	-\lambda_1 & \lambda_2 \\	
                                   \lambda_1 & \lambda_2
               \end{array}\right) u ,
\end{eqnarray}
with initial condition $u_0 = (1,1)$ on the interval $[0,T]$.

The analytical solution is given by:
\begin{eqnarray}
u(t) =  \left(\begin{array}{cc}	c_1 - c_2\exp{(-(\lambda_1+\lambda_2)t)} \\	
                                  \frac{\lambda_1}{\lambda_2}c_1+c_2\exp{(-(\lambda_1+\lambda_2)t)}
               \end{array}\right) ,
\end{eqnarray}
where
\begin{eqnarray}
 c_1 = \frac{2}{1+\frac{\lambda_1}{\lambda_2}}\qquad,\qquad 
         c_2 = \frac{1-\frac{\lambda_1}{\lambda_2}}{1+\frac{\lambda_1}{\lambda_2}} .
\end{eqnarray}
We split our linear operator into two operators by setting:
\begin{eqnarray}
\frac{\partial u(t)}{\partial t} =  \left(\begin{array}{cc}	-\lambda_1 & 0 \\	\lambda_1 & 0\end{array}\right) u
       +\left(\begin{array}{cc}	0 & \lambda_2 \\	0 & -\lambda_2 \end{array}\right) u  .
\end{eqnarray}

We choose $\lambda_1 = 0.25$ and $\lambda_2 = 0.5$ on the interval [0,1].\\[2ex]
We therefor have the operators:
\begin{eqnarray} 
A = \left(\begin{array}{cc}	-0.25 & 0 \\	0.25 &0\end{array}\right) 
       \qquad,\qquad B = \left(\begin{array}{cc}	0 & 0.5 \\	0 & -0.5 \end{array}\right) .
\end{eqnarray}

For the integration method we use a time-step size of $h = 10^{-3}$.\\[2ex]
As initialization of our iterative method we use $c_{-1} \equiv 0$\\[2ex]
From the examples one can see that the order increases by one per 
iteration step.

In Tables \ref{table1}- \ref{table3} we apply the different integration rules
to our iterative scheme.
\begin{table}[h]
\begin{center}
\begin{tabular}{||c|c||c|c||}
\hline \hline
Iterative & Number of & $err_1$  & $err_2$  \\
Steps & splitting-partitions &   &   \\
\hline
   2 &    1 &  4.5321e-002 & 4.5321e-002 \\ 
   2 &   10 &  3.9664e-003 & 3.9664e-003 \\ 
   2 &  100 &  3.9204e-004 & 3.9204e-004 \\ \hline
   3 &    1 &  7.6766e-003 & 7.6766e-003 \\ 
   3 &   10 &  6.6383e-005 & 6.6383e-005 \\ 
   3 &  100 &  6.5139e-007 & 6.5139e-007 \\ \hline
   4 &    1 &  4.6126e-004 & 4.6126e-004 \\ 
   4 &   10 &  4.1883e-007 & 4.1883e-007 \\ 
   4 &  100 &  5.9520e-009 & 5.9521e-009 \\ \hline
   5 &    1 &  4.6828e-005 & 4.6828e-005 \\ 
   5 &   10 &  1.3954e-009 & 1.3953e-009 \\ 
   5 &  100 &  5.5352e-009 & 5.5351e-009 \\ \hline
   6 &    1 &  1.9096e-006 & 1.9096e-006 \\ 
   6 &   10 &  5.5527e-009 & 5.5528e-009 \\ 
   6 &  100 &  5.5355e-009 & 5.5356e-009 \\ \hline \hline
\end{tabular}
\caption{\label{table1} Numerical results for the first example
with the iterative splitting method and the second-order Trapezoidal rule.}
\end{center}   
\end{table} 

\begin{table}[h]
\begin{center}
\begin{tabular}{||c|c||c|c||}
\hline \hline
Iterative & Number of & $err_1$  & $err_2$  \\
Steps & splitting-partitions &   &   \\
\hline
   2 &    1  & 4.5321e-002 & 4.5321e-002 \\ 
   2 &   10  & 3.9664e-003 & 3.9664e-003 \\ 
   2 &  100  & 3.9204e-004 & 3.9204e-004 \\ \hline
   3 &    1  & 7.6766e-003 & 7.6766e-003 \\ 
   3 &   10  & 6.6385e-005 & 6.6385e-005 \\ 
   3 &  100  & 6.5312e-007 & 6.5312e-007 \\ \hline
   4 &    1  & 4.6126e-004 & 4.6126e-004 \\ 
   4 &   10  & 4.1334e-007 & 4.1334e-007 \\ 
   4 &  100  & 1.7864e-009 & 1.7863e-009 \\ \hline
   5 &    1  & 4.6833e-005 & 4.6833e-005 \\ 
   5 &   10  & 4.0122e-009 & 4.0122e-009 \\ 
   5 &  100  & 1.3737e-009 & 1.3737e-009 \\ \hline
   6 &    1  & 1.9040e-006 & 1.9040e-006 \\ 
   6 &   10  & 1.4350e-010 & 1.4336e-010 \\ 
   6 &  100  & 1.3742e-009 & 1.3741e-009 \\ \hline\hline
\end{tabular}
\caption{\label{table2} Numerical results for the first example
with the iterative splitting method and third order BDF3.}
\end{center}   
\end{table} 

\begin{table}[h]
\begin{center}
\begin{tabular}{||c|c||c|c||}
\hline \hline
Iterative & Number of & $err_1$  & $err_2$  \\
Steps & splitting-partitions &   &   \\
\hline
   2 &    1  & 4.5321e-002 & 4.5321e-002 \\ 
   2 &   10  & 3.9664e-003 & 3.9664e-003 \\ 
   2 &  100  & 3.9204e-004 & 3.9204e-004 \\ \hline
   3 &    1  & 7.6766e-003 & 7.6766e-003 \\ 
   3 &   10  & 6.6385e-005 & 6.6385e-005 \\ 
   3 &  100  & 6.5369e-007 & 6.5369e-007 \\ \hline
   4 &    1  & 4.6126e-004 & 4.6126e-004 \\ 
   4 &   10  & 4.1321e-007 & 4.1321e-007 \\ 
   4 &  100  & 4.0839e-010 & 4.0839e-010 \\ \hline
   5 &    1  & 4.6833e-005 & 4.6833e-005 \\ 
   5 &   10  & 4.1382e-009 & 4.1382e-009 \\ 
   5 &  100  & 4.0878e-013 & 4.0856e-013 \\ \hline
   6 &    1  & 1.9040e-006 & 1.9040e-006 \\ 
   6 &   10  & 1.7200e-011 & 1.7200e-011 \\ 
   6 &  100  & 2.4425e-015 & 1.1102e-016 \\ \hline\hline
\end{tabular}
\caption{\label{table3} Numerical results for the first example
with the iterative splitting method and fourth order Gau{\ss} RK.}
\end{center}   
\end{table}

\begin{figure}[h]
\begin{center}
\includegraphics[width=8cm]{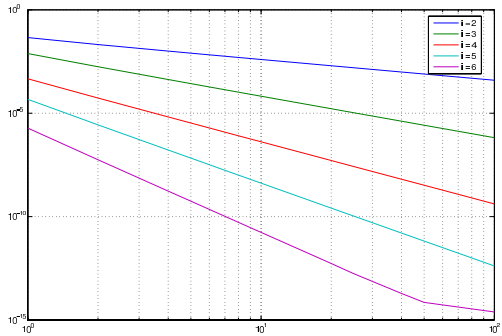}
\end{center}
\caption{\label{fig1} Convergence rates from two to six iterations.}
\end{figure}

\begin{remark}
Here we see the benefit of higher quadrature rules
in combination with the iterative operator splitting scheme, 
see Figure \ref{fig1}.
We obtain the best result with a fourth order Gauss Runge-Kutta method.
Such improved quadrature rules and the expansion of the 
integral formulation show that our method has considerable 
computational benefits.  
\end{remark}

In the next example we deal with a Schr\"odinger equations.

\subsection{Radial Sch\"odinger equation (highly nonlinear)}

We consider the radial Schr\"odinger equation\\
\begin{eqnarray}
  \frac{\partial^2 u}{\partial r^2} = f(r, E) u(r)
  \la{har}
\end{eqnarray}
where
\begin{eqnarray}
  f(r, E) = 2 V(r) - 2 E + \frac{l(l + 1)}{r^2} \; ,
\end{eqnarray}
If we re-label $r\rightarrow t$ and $u(r)\rightarrow q(t)$, (\ref{har}) can be viewed as
harmonic oscillator with a time-dependent spring constant
\be
k(t,E)=-f(t,E)
\ee
and  Hamiltonian
\begin{eqnarray}
H=\frac12 p^2+\frac12 k(t,E)q^2.
\end{eqnarray}

We compare different splitting methods with our scheme, which is 
related to a Suzuki's expansion, see \cite{suzu93}.

In Figure \ref{part_1}, we present the comparison between
fractional step (FR), Runge-Kutta Nystr\"om (RKN), standard
Magnus expansion (M), improved Magnus expansion (BM) and
Suzuki's expansion (C).
\begin{figure}[ht]
\begin{center} 
\includegraphics[width=10.0cm,angle=-0]{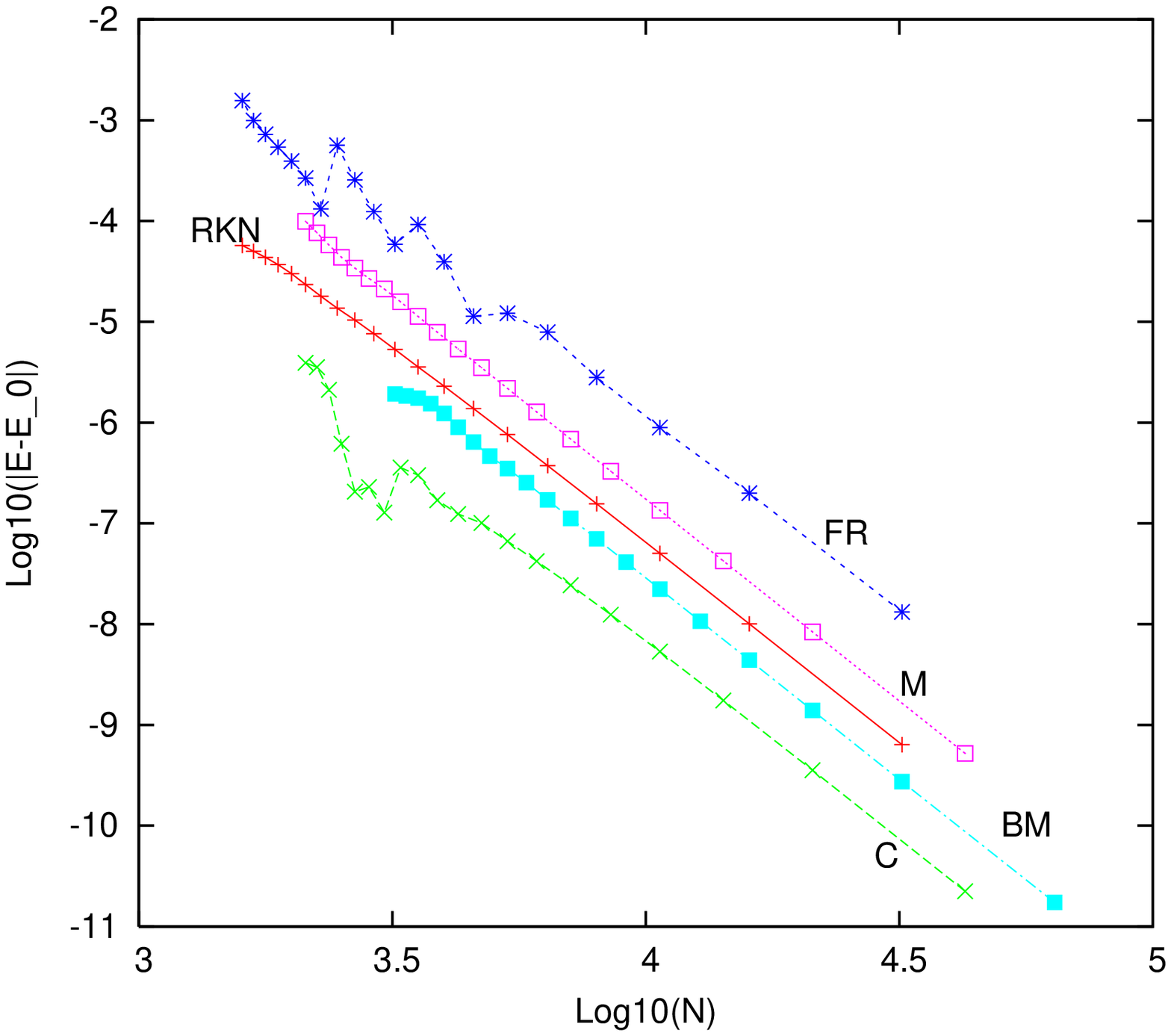} 
\end{center}
\caption{\label{part_1} Comparison between different operator-splitting method.}
\end{figure}

Here we see the benefit of the iterative 
operator-splitting method, which can be seen as a modified Suzuki's expansion
method. 

\begin{remark}
The benefit of higher quadrature rules in combination 
with the iterative operator splitting scheme
is related to Suzuki's expansion. We applied our scheme and 
obtain the best result with a fourth order method.
Such improvements based on quadrature rules, expansion of integral
formulations show that our method has considerable computational benefits.  
\end{remark}

\section{Conclusions and Discussions }
\label{conc}

We have presented an iterative operator-splitting method as
competitive method to compute split-able differential equations.
On the basis of integral formulation of the iterative 
scheme, we analyze the assumptions of the method and
its local error for unbounded operators.
Under weak assumptions we can prove the higher-order error estimates.
Numerical examples confirm the method's application to 
differential equations and to complicated Schr\"odinger equations.
In the future we will focus on the development of improved
operator-splitting methods with respect to their application
in nonlinear differential equations.

\bibliographystyle{plain}

\end{document}